\documentclass{CSML}

\def\dOi{12(3:9)2016}
\lmcsheading%
{\dOi}
{1--24}
{}
{}
{Jul.~23, 2015}
{Sep.~13, 2016}
{}

\ACMCCS{[{\bf Theory of computation}]: Logic---Constructive
  mathematics\,/\,Type theory}

\usepackage{url}
\usepackage{breakurl}
\usepackage[breaklinks]{hyperref}
\usepackage{mathtools,amssymb}
\usepackage{bigfoot}

\usepackage{amsthm,cleveref}
\crefformat{section}{\S#2#1#3}
\Crefformat{section}{Section~#2#1#3}
\crefrangeformat{section}{\S\S#3#1#4--#5#2#6}
\Crefrangeformat{section}{Sections~#3#1#4--#5#2#6}
\crefmultiformat{section}{\S\S#2#1#3}{ and~#2#1#3}{, #2#1#3}{ and~#2#1#3}
\Crefmultiformat{section}{Sections~#2#1#3}{ and~#2#1#3}{, #2#1#3}{ and~#2#1#3}
\crefrangemultiformat{section}{\S\S#3#1#4--#5#2#6}{ and~#3#1#4--#5#2#6}{, #3#1#4--#5#2#6}{ and~#3#1#4--#5#2#6}
\Crefrangemultiformat{section}{Sections~#3#1#4--#5#2#6}{ and~#3#1#4--#5#2#6}{, #3#1#4--#5#2#6}{ and~#3#1#4--#5#2#6}

\theoremstyle{plain}
\crefname{thm}{Theorem}{Theorems}
\newtheorem{lemma}[thm]{Lemma}
\crefname{lemma}{Lemma}{Lemmas}
\newtheorem{corollary}[thm]{Corollary}
\crefname{corollary}{Corollary}{Corollaries}

\crefname{sch}{Scholium}{Scholia}

\theoremstyle{definition}
\newtheorem{defn}[thm]{Definition}
\crefname{defn}{Definition}{Definitions}
\newtheorem{rmk}[thm]{Remark}
\crefname{rmk}{Remark}{Remarks}
\newtheorem{rmks}[thm]{Remarks}
\crefname{rmks}{Remarks}{Remarks}
\newtheorem{ex}[thm]{Example}
\crefname{ex}{Example}{Examples}

\usepackage[all]{xy}

\def\id{\mathsf{id}}
\def\refl{\mathsf{refl}}
\def\ap{\mathsf{ap}}
\def\mltt{\textsc{mltt}}
\def\zfc{\textsc{zfc}}
\def\uip{\textsc{uip}}
\def\oo{\ensuremath{\infty}}
\def\inl{\mathsf{inl}}
\def\inr{\mathsf{inr}}
\def\bN{\mathbb{N}}
\def\bone{\mathbf{1}}
\def\btwo{\mathbf{2}}
\def\type{\mathsf{Type}}
\def\Retr{\mathsf{Retr}}
\def\Split{\mathsf{Split}}
\def\QIdem{\mathsf{QIdem}}
\def\Idem{\mathsf{Idem}}
\def\PIdem{\mathsf{PIdem}}
\def\split{\mathsf{split}}
\def\uli{\mathsf{uli}}
\def\jdeq{\equiv}

\def\baut{B\mathsf{Aut}}
\def\aut{\mathsf{Aut}}
\newcommand{\defeq}{\vcentcolon\equiv}  % A judgmental equality currently being defined
\newcommand{\brck}[1]{\mathopen{}\left\Vert #1\right\Vert\mathclose{}}
\newcommand{\bproj}[1]{\mathopen{}\left|#1\right|\mathclose{}}

\newcommand{\ct}{%
  \mathchoice{\mathbin{\raisebox{0.5ex}{$\displaystyle\centerdot$}}}%
             {\mathbin{\raisebox{0.5ex}{$\centerdot$}}}%
             {\mathbin{\raisebox{0.25ex}{$\scriptstyle\,\centerdot\,$}}}%
             {\mathbin{\raisebox{0.1ex}{$\scriptscriptstyle\,\centerdot\,$}}}
}

\begin{document}

\title[Idempotents in intensional type theory]{Idempotents in intensional type theory} % [Idempotents in type theory]

\author[M.~Shulman]{Michael Shulman}
\address{University of San Diego\\5998 Alcala Park\\San Diego, CA 92110}
\email{shulman@sandiego.edu}
\thanks{This material is based on research sponsored by The United States Air Force Research Laboratory under agreement number FA9550-15-1-0053.  The U.S.~Government is authorized to reproduce and distribute reprints for Governmental purposes notwithstanding any copyright notation thereon.  The views and conclusions contained herein are those of the author and should not be interpreted as necessarily representing the official policies or endorsements, either expressed or implied, of the United States Air Force Research Laboratory, the U.S.~Government, or Carnegie Mellon University.}	

%% required for running head on odd and even pages, use suitable
%% abbreviations in case of long titles and many authors:

%% mandatory lists of keywords and classifications:

%\ACMCCS{\textbf{Theory of computation~Type theory}; \textit{Theory of computation~Constructive mathematics}; \textit{Software and its engineering~Functional languages}}

%% the abstract has to PRECEED the command \maketitle:
%% be sure not to issue the \maketitle command twice!

\begin{abstract}
  We study idempotents in intensional Martin-L\"of type theory, and in particular the question of when and whether they split.
  We show that in the presence of propositional truncation and Voevodsky's univalence axiom, there exist idempotents that do not split; thus in plain MLTT not all idempotents can be proven to split.
  On the other hand, assuming only function extensionality, an idempotent can be split if and only if its witness of idempotency satisfies one extra coherence condition.
  Both proofs are inspired by parallel results of Lurie in higher category theory, showing that ideas from higher category theory and homotopy theory can have applications even in ordinary MLTT.

  Finally, we show that although the witness of idempotency can be recovered from a splitting, the one extra coherence condition cannot in general; and we construct ``the type of fully coherent idempotents'', by splitting an idempotent on the type of partially coherent ones.
  Our results have been formally verified in the proof assistant Coq.
\end{abstract}

\maketitle

\section{Introduction}
\label{sec:intro}

In December 2014 Mart\'in Escard\'o asked me whether idempotents split in Martin-L\"of type theory (\mltt).
This paper is a long-winded answer.

Usually, an \emph{idempotent} means a function (necessarily an endofunction) that is equal to its composite with itself, $f\circ f= f$.
In \mltt, using the propositions-as-types methodology, we might naturally take this to mean a function $f:X\to X$, for some type $X$, together with a \emph{witness of idempotency} $I : \prod_{x:X} (f(f(x)) = f(x))$, where ``$=$'' denotes the identity type.
(If we assume function extensionality, as we often will, then to give $I$ is equivalent to giving $I' : f\circ f = f$.)

A \emph{splitting} of an idempotent $f$ on $X$ consists of functions $r:X\to A$ and $s:A\to X$ such that $r\circ s = \id_A$ and $s\circ r = f$.
In \zfc\ set theory, an idempotent always has a splitting with $A = \{ x\in X \mid f(x)=x \}$, where $s$ is the inclusion and $r$ is the corestriction of $f$.
This suggests that in \mltt\ we ought to consider $A = \sum_{x:X} (f(x)=x)$, with $s$ the first projection $s(x,p) \defeq x$ and $r$ defined by $r(x) \defeq (f(x),I(x))$.
However, as Mart\'in observed, this does not work in general:

\begin{ex}\label{thm:split-set}
  Let $X$ be any type and let $f=\id_X$ be its identity function, with $I$ the obvious witness defined by $I(x) \defeq \refl_x$.
  Then with the above-defined $A$, $s$, and $r$, we have $\prod_{a:A} (r(s(a))=a)$ if and only if
 % is logically equivalent to % (i.e.\ implies and is implied by)
 %  $\prod_{x:X} \prod_{p:x=x} ((x,\refl_x)=_A (x,p))$.
 %  However, this type is inhabited exactly when
 %  $\prod_{x:X} \prod_{p:x=x} (\refl_x = p)$
 %  is, and this is to say that
  $X$ satisfies Uniqueness of Identity Proofs (\uip).
  (We will prove this in \cref{sec:some-split}.)
  Since \mltt\ does not prove that all types satisfy \uip, neither can it prove that this construction always splits an idempotent.
\end{ex}

Now, if we were wondering whether \mltt\ proved some theorem and we had found that the obvious proof of some theorem used a classical axiom such as the law of excluded middle, then it would be natural to seek for counterexamples in nonclassical models (such as topological or realizability models) or disproofs from nonclassical axioms (such as strong Church's thesis or Brouwerian continuity principles).
Similarly, having found that the obvious way to split idempotents depends on \uip, it is natural to seek counterexamples in models that violate \uip\ or disproofs from axioms that contradict it.

This leads us into the recently discovered realm of Homotopy Type Theory and Univalent Foundations~\cite{apw:vvu-hott,awodey:tt-and-htpy,pw:hottvvuf,hottbook}.
Models which violate \uip, such as the Hofmann--Streicher groupoid model~\cite{hs:gpd-typethy} and Voevodsky's simplicial set model~\cite{klv:ssetmodel}, tend to be based on the idea that types are \emph{homotopy spaces} or \emph{\oo-groupoids}.
The principal known axiom that contradicts \uip\ --- Voevodsky's univalence axiom --- is also based on this idea.

This suggests that when seeking inspiration from classical mathematics, instead of \zfc\ set theory we should look to homotopy theory and \oo-groupoid theory.
In these fields, an important role is played by \emph{homotopy coherence}.
When a structure satisfies some property ``up to homotopy'', for many purposes it is not enough to simply have such a homotopy; often one requires this homotopy to satisfy some natural axiom(s) at the next dimension up --- and that only up to homotopy, a homotopy that in turn satisfies its own axioms up to even higher homotopy, and so on to infinity.

For instance, instead of a \emph{group} we may consider an \emph{\oo-group} (a.k.a.\ ``grouplike $A_\oo$-space'').
This is a space $X$ with a multiplication $m:X\times X\to X$ that is, among other things, associative but only up to homotopy: for any $x,y,z\in X$ instead of $m(m(x,y),z)$ and $m(x,m(y,z))$ being equal, they are connected by a path depending continuously on $x,y,z$.
These paths are then required to satisfy a further property: for any $x,y,z,w\in X$ there is a pentagon that can be built out of these paths, and we require that there be a continuous way to ``fill in'' that pentagon inside $X$.
From those filled pentagons one can then construct the boundary of a certain polyhedron, which we require to have a continuous filler, and so on.
If we stop at any finite stage, we obtain a much more poorly-behaved notion.

Now, under the homotopical interpretation of \mltt, a witness of idempotency $I : \prod_{x:X} (f(f(x)) = f(x))$ corresponds to a \emph{homotopy} from $f\circ f$ to $f$.
Thus, from a homotopy-theoretic point of view, it is natural to expect that $I$ itself would not be enough to obtain a well-behaved notion of ``idempotent'' (such as, for instance, one that can be split): we should ask it to satisfy a further property analogous to filling the pentagon, and that filler should itself satisfy a higher axiom up to homotopy, and so on.
In the context of \oo-categories, such a definition of \emph{fully-coherent idempotent} has been given by Lurie in~\cite[\S4.4.5]{lurie:higher-topoi}, along with proofs that every fully-coherent idempotent splits and every split idempotent is fully-coherent.

Unfortunately, there is a well-known problem with representing such fully-coherent structures in \mltt: on the face of it they seem to require a tower of infinitely many terms, each dependent on the previous ones, which is not something that can be defined as a single object in \mltt.
% It is true, though, that in some particular cases one can find finite representations for such apparently infinite structures.
% For instance, the \oo-categorical notion of \emph{equivalence} involves \emph{a priori} not just two functions $f:X\to Y$ and $g:Y\to X$ and homotopies $\epsilon : f\circ g \sim \id$ and $\eta : \id\sim g\circ f$ to identities, but higher coherence homotopies \emph{ad infinitum}.
% However, it is now well-known that there are plenty of ways to define ``equivalence'' using only the usual finitary operations in \mltt\ such that ``the type of equivalences'' yields the correct homotopical object (the first was due to Voevodsky); generally these involve one additional coherence datum beyond $(f,g,\epsilon,\eta)$.
% Moreover, given only $(f,g,\epsilon,\eta)$, it is possible to modify one of $\epsilon$ or $\eta$ in order that this additional coherence datum can be found.
% However, this nice situation with equivalences seems to be the exception rather than the rule as far as infinite structures go.
This is somewhat disheartening for the project of splitting idempotents.
However, it's important not to read more into the results of Lurie cited above than they say.
They do say that if a function $f$ can be written as $s\circ r$ where $r\circ s = \id$, then $f$ admits a ``coherent system of idempotence data''.
They \emph{don't} say that if $f$ is idempotent with a specified homotopy $I$, and $f$ splits, then $I$ must \emph{itself} admit an extension to a coherent system of idempotence data.
Therefore:

\begin{itemize}
\item Even though a split idempotent automatically gives rise to an infinite system of coherence data, it doesn't follow that in order to \emph{construct} a splitting we would necessarily need to \emph{give} an infinite system of coherence.
\item It's not too hard to give examples of homotopies $I$ that are not coherent, but it's rather less obvious how to give an example of an incoherent idempotent for which there doesn't exist some \emph{other} homotopy that is coherent.
\end{itemize}

\noindent Fortunately, Lurie has already addressed these questions as well (still in the \oo-categorical context).
In~\cite[Warning 1.2.4.8]{lurie:ha} he gave an example of an incoherent idempotent that does \emph{not} split, and in~\cite[Lemma 7.3.5.14]{lurie:ha} he showed that to construct a splitting, \emph{one} additional coherence datum suffices.
Inspired by these results, we will set out to transfer them to \mltt, as follows:
\begin{enumerate}
\item Assuming propositional truncation and the univalence axiom, we can adapt Lurie's counterexample to show that a single witness of idempotency $I : \prod_{x:X} (f(f(x)) = f(x))$ is insufficient to construct a splitting.
  In fact, our construction is slightly simpler than Lurie's, and involves an object familiar to constructive mathematicians: the Cantor space $2^{\bN}$.
\item However, under the weaker assumption of function extensionality, % (which follows from univalence and also from propositional truncation),
  we can adapt Lurie's construction to show that if we also have $J: \prod_{x:X} (\ap_f(I(x)) = I(f(x)))$, then we can construct a splitting.
  (Here $\ap_f$ denotes the action of $f$ on witnesses of equality; sometimes it is called $\mathsf{resp}$.)
  Our construction is actually the dual of Lurie's: we use a limit where he uses a colimit.
  A colimit would probably also work, but would require further assumptions on type theory for its construction and well-behavedness.
\end{enumerate}
Note that the latter positive result does \emph{not} require the univalence axiom.
So although inspired by higher category theory, we obtain a result that should be of interest even in pure intensional \mltt.

Based on these results, we propose that, as in higher category theory, the unadorned word \emph{idempotent} should not be used for the ``incoherent'' notion that includes only a single witness $I$.
Instead we will call the pair $(f,I)$ a \textbf{pre-idempotent}.

One might think that the triple $(f,I,J)$ ought to deserve the name ``idempotent'', since although it does not include all the higher coherence data, we have seen that it does suffice to construct a splitting. % analogous to the good definitions of ``equivalence'' mentioned above.
However, this is not the case.
It is true, in the \oo-categorical world, that a splitting induces a fully coherent idempotent in Lurie's sense, and hence so does a triple $(f,I,J)$; but nothing guarantees that the resulting coherent idempotent is an extension of $(f,I,J)$ itself.
In fact, we will show in \mltt\ that it is an extension of $(f,I)$, but not in general of $J$: assuming univalence and propositional truncation, there exist choices for $J$ that are not coherentifiable at all.
For these reasons, we will instead call a triple $(f,I,J)$ a \textbf{quasi-idempotent}; it is analogous to the ``incoherent, but coherentifiable, equivalences'' that
% quadruples $(f,g,\epsilon,\eta)$ from which one can construct a coherent equivalence, which
in~\cite{hottbook} are called \emph{quasi-inverses}.

At this point we may wonder whether there is \emph{any} way to define the word ``idempotent'' in \mltt\ in a way that will translate to the correct notion homotopically.
There is one answer that is somewhat ``cheap'': by~\cite[Corollary 4.4.5.14]{lurie:higher-topoi}, in an \oo-category the space of idempotents on an object $X$ is equivalent to the space of \emph{retractions} of $X$, meaning quadruples $(A,r,s,H)$ where $r:X\to A$ and $s:A\to X$ and $H$ is a homotopy $r\circ s \sim \id_A$.
The latter can be defined in \mltt\ (with universes) as
\[ \textstyle \Retr(X) \defeq \sum_{A:\type} \sum_{r:X\to A} \sum_{s:A\to X} \prod_{a:A} (r(s(a))=a), \]
and if we assume the univalence axiom, it will have the correct homotopy type.
Thus, we could define an \emph{idempotent} on $X$ to be an inhabitant of this type.

Of course, this would be rather unsatisfying: we expect an ``idempotent'' to consist of a map $f:X\to X$ equipped with some kind of structure, and we expect the construction of its splitting to be a nonvacuous operation.
Moreover, it has an actual technical drawback as well: since it involves a sum over a universe $\type$, it lives in a universe one higher than that of $X$.

Both of these problems can be solved with the following trick.
If we define the type of quasi-idempotents in the expected way:
\[ \textstyle\QIdem(X) \defeq \sum_{f:X\to X} \sum_{I : \prod_{x:X} (f(f(x)) = f(x))} \prod_{x:X} (\ap_f(I(x)) = I(f(x))) \]
then the above splitting construction yields a map
\[ \split : \QIdem(X) \to \Retr(X). \]
On the other hand, since every retraction induces a coherent idempotent, we have a map
\[ \uli : \Retr(X) \to \QIdem(X) \]
and these two maps can be shown to exhibit $\Retr(X)$ itself as a retract of $\QIdem(X)$.
Therefore, the composite $\uli \circ \split$ is a quasi-idempotent on $\QIdem(X)$, and if we construct its splitting as above, we obtain a type equivalent to $\Retr(X)$.
This splitting type is what we propose as the definition of \textbf{(fully coherent) idempotent}: it has the correct homotopy type; it is by construction an equipping of an endomap with data (indeed, infinitely many data, encoded internally by way of the natural numbers type); and it lies in the same universe as $X$.

The plan of the paper is as follows.
In \cref{sec:notation} we recall some notation and terminology from~\cite{hottbook}.
In \cref{sec:some-split} we ease into the study of idempotents by considering several hypotheses (due to Mart\'in Escard\'o) under which pre-idempotents can be split.
The next two sections contain the main results:
in \cref{sec:some-dont} we give our example of a pre-idempotent that admits no splitting (assuming propositional truncation and the univalence axiom), and
in \cref{sec:all-quasi-do} we construct a splitting of any quasi-idempotent (assuming function extensionality).

The remaining sections are concerned with the more technical coherence questions.
In \cref{sec:splitting-retraction} we show that $\split$ exhibits $\Retr(X)$ as a retract of $\QIdem(X)$.
In \cref{sec:splitting-not-equiv} we show that this retraction is not an equivalence, and conclude that although the underlying pre-idempotent of a quasi-idempotent can be recovered from its splitting, the coherence datum $J$ cannot in general be.
In \cref{sec:baut-baut-bool} we complete a proof from \cref{sec:splitting-not-equiv} that requires a lengthy analysis of some classifying spaces in type theory.
Finally, in \cref{sec:coherent-idempotents} we define the type of fully-coherent idempotents, and in \cref{sec:conclusions} we conclude with some remaining open problems.

Throughout, we will argue in the informal style of~\cite{hottbook}, and we will make use of the basic results from Chapters 1--4 thereof.
However, all the main results of this paper have also been formally verified in the proof assistant Coq, using the Homotopy Type Theory library~\cite{hottcoq}, and are available as part of that library.
% Because of this, we feel more justified in omitting some tedious details: in addition to the usual choice of trying to recreate them him- or herself, the reader has the options of simply trusting the formalization, or of reading or stepping through it to understand the calculations if desired.
As of the date of publication, the correspondence between sections of this paper and files in the library is:
\begin{center}
  \begin{tabular}{c|l}
    Section & Library File\\ \hline
    \cref{sec:some-split} & \texttt{Idempotents.v}\\
    \cref{sec:some-dont} & \texttt{Spaces/BAut/Cantor.v}\\
    \crefrange{sec:all-quasi-do}{sec:splitting-not-equiv} & \texttt{Idempotents.v}\\
    \cref{sec:baut-baut-bool} & \texttt{Spaces/BAut.v} and \texttt{Spaces/BAut/Bool.v}\\
    \cref{sec:coherent-idempotents} & \texttt{Idempotents.v}
  \end{tabular}
\end{center}
The \texttt{idempotents-paper} git tag records this version of the library.

\section{Some notation and terminology}
\label{sec:notation}

For the most part, we adopt the notation and terminology of~\cite{hottbook}.
We write $\prod_{x:A} B(x)$ and $\sum_{x:A}B(x)$ for dependent product and sum as usual in \mltt, with their non-dependent special cases $A\to B$ and $A\times B$.
We write the identity type of two elements $x,y:A$ as $x=_Ay$, or usually just $x=y$; its canonical elements are $\refl_x : x=x$.
We write $x\jdeq y$ for a judgmental equality, and $a\defeq b$ if $a$ is currently being defined to equal $b$.

A type $A$ is called a \textbf{mere proposition} if we have $\prod_{x,y:A} (x=y)$.
In other words, $\mathsf{isprop(A)} \coloneqq \prod_{x,y:A} (x=y)$.
It is said to \textbf{be a set}, or to \textbf{satisfy} \uip\ (Uniqueness of Identity Proofs), if $\prod_{x,y:A} \mathsf{isprop}(x=y)$, or equivalently $\prod_{x:X} \prod_{p:x=x} (p = \refl_x)$.

For functions $f,g:A\to B$, we write $f\sim g$ for the type $\prod_{x:A} (f(x)=g(x))$, and call it the type of \textbf{homotopies} from $f$ to $g$.
The \textbf{function extensionality} axiom, which we will almost always have available (either by explicit assumption, or as a consequence of some other assumption), says that this type is equivalent (see below) to the identity type $f=_{A\to B}g$.

For types $A$ and $B$, we write $A \simeq B$ for the type of \textbf{equivalences} from $A$ to $B$.
This is defined as $\sum_{f:A\to B} \mathsf{isequiv}(f)$, where $\mathsf{isequiv}(f)$ is any one of a number of well-behaved definitions, the first of which was due to Voevodsky; see~\cite[Chapter 4]{hottbook} for details.
The important properties are that $\mathsf{isequiv}(f)$ if and only if $\sum_{g:B\to A} (f\circ g\sim \id) \times (g\circ f \sim \id)$ (we generally use the ``if'' direction of this to construct equivalences), and that $\mathsf{isequiv}(f)$ is a mere proposition.
There is a canonical map $(A=B) \to (A\simeq B)$, and Voevodsky's \textbf{univalence axiom} says that this map is itself an equivalence.

The \textbf{propositional truncation} is, when assumed, a rule associating to every type $A$ a type $\brck{A}$ which is a mere proposition, and a map $\bproj{-} : A\to \brck A$, such that any function from $A$ to a mere proposition factors judgmentally through $\brck A$.
In other words, if $B$ is a mere proposition and $f:A\to B$, then there exists $g:\brck A \to B$ such that $f(a) \jdeq g(\bproj{a})$ for all $a:A$.
We sometimes pronounce $\brck{A}$ as ``\textbf{merely $A$}'', e.g.\ if we have an element of $\brck{A\simeq B}$ we say that $A$ and $B$ are merely equivalent.
Both univalence and propositional truncation imply function extensionality; the former is due to Voevodsky (see e.g.~\cite[\S4.9]{hottbook}) and the latter to~\cite[Corollary 8.3]{keca:anon}.

With homotopy-theoretic intuition in mind, elements of identity types (i.e.\ witnesses of equality) are sometimes called \textbf{paths}.
For $p:x=_A y$ and $q:y=_A z$, we have $p\ct q: x=_A z$ (a witness of transitivity) and $p^{-1} : y=_A x$ (a witness of symmetry), defined using the eliminator of the identity type.
Similarly, if $f:A\to B$, we have $\ap_f(p) : f(x) =_B f(y)$, and this operation is functorial (up to propositional equality) in two ways: $\ap_g \circ \ap_f = \ap_{g\circ f}$, and $\ap_f(p\ct q) = \ap_f(p) \ct \ap_f(q)$.

We will frequently use the fact that homotopies between functions satisfy a \emph{naturality} property~\cite[Lemma 2.4.3]{hottbook}.
Specifically, given $g,h:B\to C$ and $L:g\sim h$, for any $b_1,b_2:B$ and $p:b_1=b_2$, we have
\[ \ap_g(p) \ct L(b_2) = L(b_1) \ct \ap_h(p). \]

The other important facts we will use from~\cite{hottbook} are the theorems from its Chapter 2 that characterize the identity types of different type formers (sometimes requiring univalence and function extensionality).
For instance, the type $(a_1,b_1)=_{A\times B}(a_2,b_2)$ is equivalent to $(a_1=_A a_2) \times (b_1 =_B b_2)$, i.e.\ two ordered pairs are equal just when their components are.
In most cases these results are reasonably intuitive.

\section{Some pre-idempotents that split}
\label{sec:some-split}

As suggested in the introduction, we define:

\begin{defn}
  A \textbf{pre-idempotent} is an endofunction $f:X\to X$ equipped with a witness of idempotency $I : f\circ f \sim f$.
\end{defn}

\begin{defn}
  A \textbf{retract} of a type $X$ consists of a type $A$, functions $s:A\to X$ and $r:X\to A$, and a homotopy $H: r\circ s \sim \id_A$.
  A \textbf{splitting} of an endofunction $f:X\to X$ is a retraction $(A,r,s,H)$ together with a homotopy $K: s \circ r \sim f$.
\end{defn}

The following is fairly obvious.

\begin{lemma}\label{thm:split-pre}
  If $f$ has a splitting, then it is pre-idempotent.
\end{lemma}
\begin{proof}
  Clearly anything homotopic to a pre-idempotent is pre-idempotent, so it suffices to show that if we have a retraction $(A,r,s,H)$ then $s\circ r$ is pre-idempotent.
  In this case, for any $x:X$, we can define $I(x)\defeq \ap_s(H(r(x))) : s(r(s(r(x)))) = s(r(x))$.
\end{proof}

We can also show easily that splittings are essentially unique in at least a weak sense.

\begin{lemma}\label{thm:split-uniq}
  Suppose $f:X\to X$ has two splittings $(A,s,r,H,K)$ and $(A',s',r',H',K')$.
  Then $A\simeq A'$.
\end{lemma}
\begin{proof}
  We have two functions $r's:A\to A'$ and $rs':A'\to A$, and their composites are homotopic to identities:
  \[ r'srs' \sim r'fs' \sim r's'r's' \sim \id_{A'} \]
  and similarly $rs'r's \sim \id_A$.
\end{proof}

We expect that a split endofunction is not only pre-idempotent, but fully-coherently idempotent.
As remarked in the introduction, it is difficult to define fully-coherent idempotents in type theory, but we can at least define the next step of coherence.

\begin{defn}
  A \textbf{quasi-idempotent} is a pre-idempotent $(f,I)$ together with a witness of coherence $J: \prod_{x:X} (\ap_f(I(x)) = I(f(x)))$.
\end{defn}

\begin{lemma}\label{thm:split-quasi}
  If $f$ has a splitting, then it is quasi-idempotent.
\end{lemma}
\begin{proof}
  As in \cref{thm:split-pre}, it suffices to show that for any retraction $(A,s,r,H)$, $s\circ r$ is quasi-idempotent.
  For this case, in \cref{thm:split-pre} we defined $I(x)\defeq \ap_s(H(r(x))) : s(r(s(r(x)))) = s(r(x))$.
  Thus, for $x:X$ the desired type of $J(x)$ is
  \[ \ap_f(\ap_s(H(r(x)))) = \ap_s(H(r(f(x)))). \]
  This is equivalent to
  \[ \ap_s(\ap_{r\circ s}(H(r(x)))) = \ap_s(H(r(s(r(x))))). \]
  Peeling off an $\ap_s$, and letting $a\defeq r(x)$, it will suffice to show that for any $a:A$ we have
  \[ \ap_{r\circ s}(H(a)) = H(r(s(a))). \]

  At first this seems like a nontrivial property of $H$.
  However, in fact it is automatic.
  For by naturality of the homotopy $H$ applied at the equality $H(a)$, we have
  \[ \ap_{r\circ s}(H(a)) \ct H(a) = H(r(s(a))) \ct H(a). \]
  Now we can cancel $H(a)$ from both sides to obtain the desired result.
\end{proof}

We now give a few conditions under which pre-idempotents can be split.
Our first observation is:

\begin{thm}\label{thm:set-split}
  If $X$ is a set, then any pre-idempotent on $X$ has a splitting.
\end{thm}
\begin{proof}
  Define $A\defeq \sum_{x:X} (f(x)=x)$, and let $s$ and $r$ be defined by $s(x,p) = x$ and $r(x) = (f(x),I(x))$.
  Now for $x:X$, we have $s(r(x)) \jdeq f(x)$ by definition; hence we can take $K(x) \defeq \refl_{f(x)}$.
  On the other hand, for $(x,p):A$ we have $r(s(x,p)) \jdeq (f(x),I(x))$; thus $H(x,p)$ must inhabit $((f(x),I(x)) = (x,p))$.
  By~\cite[Theorems 2.7.2 and 2.11.3]{hottbook}, to give an element of this type we must give $q:f(x)=x$ and $r:\ap_f(q)^{-1} \ct I(x) \ct q = p$.
  But we can take $q \defeq p$, and obtain $r$ from the assumption that $X$ is a set.
\end{proof}

Now here is our elaboration of \cref{thm:split-set}, showing that this construction cannot always work.

\begin{ex}[Escard\'o]\label{thm:split-set2}
  Let $X$ be any type and $f\defeq \id_X$, with $I(x) \defeq \refl_x$.
  Then with the above-defined $A$, $s$, and $r$, the desired type of $H(x,p)$ is equivalent to $q^{-1} \ct \refl_x \ct q = p$, and hence to $\refl_x = p$.
  If this is true for all $x:X$ and all $p:x=x$, then $X$ satisfies \uip.
\end{ex}

Escard\'o has also observed a couple of other situations in which pre-idempotents can be split.
For the first, recall from~\cite{keca:anon} that a function $f:X\to Y$ is \textbf{weakly constant} if we have a witness $\prod_{x,y:X} (f(x)=f(y))$.

\begin{thm}[Escard\'o]
  If a pre-idempotent is weakly constant, then it has a splitting.
\end{thm}
\begin{proof}
  We use the same construction as in \cref{thm:set-split}; by following the proof thereof, it remains only to construct $H$.
  However, by~\cite[Lemma 4.1]{keca:anon}\footnote{Also formalized in~\cite{hottcoq} as \verb|ishprop_fix_wconst| in \texttt{Constant.v}.}, when $f$ is weakly constant, our type $A\defeq \sum_{x:X} (f(x)=x)$ (there called $\mathsf{fix}(f)$) is a mere proposition, i.e.\ we have $\prod_{a,b:A} (a=b)$.
  This makes the construction of $H$ trivial.
\end{proof}

Conversely, it is easy to see that if an endofunction splits through a mere proposition, then it is weakly constant.

For the second, recall from~\cite{keca:anon} that a type admits a weakly constant endofunction if and only if there is some mere proposition $P$ with functions $A\to P$ and $P\to A$ (since any function that factors through a mere proposition is weakly constant, while by~\cite[Lemma 4.1]{keca:anon} if $f$ is weakly constant we can take $P\defeq \mathsf{fix}(f)$).
Moreover, if we have propositional truncation, this condition is equivalent to the existence of a map $\brck{A} \to A$, a property which one may call having \textbf{split support}.
Finally, recall from~\cite[Lemma 7.6.2]{hottbook} that a function $f:A\to B$ is said to be an \textbf{embedding} if for all $b:B$ the type $\sum_{a:A} (f(a)=b)$ is a mere proposition.

The following theorem is our first example of a definable splitting in which the splitting type $A$ is \emph{not} the obvious $\sum_{x:X} (f(x)=x)$.

\begin{thm}[Escard\'o]\label{thm:split-splitsupp}
  An endofunction $f$ has a splitting in which the section $s$ is an embedding if and only if it is pre-idempotent and the type $f(x)=x$ admits a weakly constant endofunction for all $x$.
\end{thm}

(It is arguably more natural to formulate this theorem in terms of split support.
The advantage of using weakly constant endofunctions instead is that it makes sense even in the absence of propositional truncation.)

\begin{proof}
  First suppose $f$ is pre-idempotent and each $f(x)=x$ has a weakly constant endofunction.
  Thus, for each $x$ there is a mere proposition $P_x$ and maps $u_x: (f(x)=x) \to P_x$ and $v_x : P_x \to (f(x)=x)$.
  (If we have propositional truncation, we can take $P_x \defeq \brck{f(x)=x}$, and the reader may find it easier to think about this case.)
  We define $A\defeq \sum_{x:X} P_x$, with $s(x,p) \defeq x$ and $r(x) \defeq (f(x),u_{f(x)}(I(x)))$, while $K(x) \defeq \refl_{f(x)}$ as before.
  For $H$, given $(x,p):A$ where $x:X$ and $p: P_x$, we must show that $(f(x),u_{f(x)}(I(x))) = (x,p)$, which as before amounts to giving $q:f(x)=x$ and an equality $q_*(u_{f(x)}(I(x))) = p$, where $q_* : P_{f(x)} \to P_x$ denotes transport along $q$.
  But we can define $q \defeq v_x(p)$, while the remaining equality is trivial since $P_x$ is a mere proposition.

  Now conversely, suppose $f$ has a splitting in which $s$ is an embedding; it remains to show that $f(x)=x$ has a weakly constant endofunction for all $x:X$.
  Since $s$ is an embedding, the type $\sum_{a:A} (s(a)=x)$ is a mere proposition.
  Thus, it will suffice to construct maps in both directions relating this type to $f(x)=x$, or equivalently to the type $s(r(x))=x$.
  In one direction, given $p:s(r(x))=x$, we have $(r(x),p):\sum_{a:A} (s(a)=x)$.
  In the other, given $a:A$ and $p:s(a)=x$, we have
  $ s(r(x)) = s(r(s(a))) = s(a) = x$.
\end{proof}

\begin{rmks}\ 
  \begin{enumerate}
  \item \cref{thm:set-split} is a special case of \cref{thm:split-splitsupp}: if $X$ is a set, then each type $f(x)=x$ is a mere proposition, hence trivially has a weakly constant endofunction.
    Moreover, by~\cite[Theorem 3.10]{keca:anon}, if \emph{every} type $x=_X y$ has a weakly constant endofunction, then $X$ is necessarily a set.
    However, there do exist functions on non-sets to which \cref{thm:split-splitsupp} applies; a trivial example is $X \defeq Y+\bone$ with $f(x)\defeq \inr(\mathsf{tt})$.
  \item On the other hand, there can exist retractions for which the section is not an embedding.
    For instance, any map $\bone \to X$ exhibits $\bone$ as a retract of $X$, but to say that all such maps are embeddings is just to say that $X$ is a set.
    Thus, \cref{thm:split-splitsupp} emphasizes another way in which idempotents in homotopy theory differ from idempotents in set theory, since in set theory the splitting of an idempotent \emph{always} injects into the original set.
  \item When the conditions of \cref{thm:split-splitsupp} fail, it doesn't generally mean there is any particular $x:X$ such that $f(x)=x$ does not have a weakly constant endofunction.
    It only means we cannot assert that ``$f(x)=x$ has a weakly constant endofunction for all $x:X$'', since such an assertion would imply an impossible ``naturality'' of the endofunctions.
  \item Perhaps surprisingly, none of the results in this section require even function extensionality.
  \end{enumerate}
\end{rmks}

\section{A pre-idempotent that doesn't split}
\label{sec:some-dont}

As mentioned in the introduction, it's easy to give examples of idempotence witnesses $I:f\circ f \sim f$ that cannot be extended to a coherent system of idempotence data.

\begin{ex}
  Let $X$ be any type with a point $x_0:X$ for which there exists a nontrivial $p:x_0=x_0$, i.e.\ such that $p\neq \refl_{x_0}$.
  (For instance, in the presence of the univalence axiom, we could let $X$ be the universe, with $x_0$ a type admitting a nonidentity self-equivalence, such as $\btwo$.)
  Define $f:X\to X$ by $f(x)\defeq x_0$ for all $x$, and let $I(x) \defeq p$ for all $x$.
  Then $(f,I)$ is a pre-idempotent.
  But $\ap_f(q) = \refl_{x_0}$ for all $q$, so the second-level coherence type $\prod_{x:X} (\ap_f(I(x)) = I(f(x)))$ is equivalent to $\prod_{x:X} (\refl_{x_0} = p)$, which by assumption is not inhabited.
\end{ex}

However, it's less clear how to exhibit a pre-idempotent $(f,I)$ for which there cannot exist any \emph{other} witness $I'$ that \emph{is} coherent.
For instance, in the above example, we could simply have taken $I'(x) \defeq \refl_{x_0}$.

We will describe an example inspired by that of~\cite[Warning 1.2.4.8]{lurie:ha}, but not quite identical to it.
In Lurie's example, the space $X$ is the classifying space of the group of endpoint-preserving self-homeomorphisms of the unit interval $[0,1]$.
However, the essential feature of this choice, for the purposes of the example, is that two such homeomorphisms can be shrunk by a factor of 2 and glued together to form a new such.
This is reminiscent of Freyd's universal characterization of $[0,1]$ (see e.g.~\cite[D4.7.17]{ptj:elephant2}), but in fact it can be completely divorced from the topology.
Thus, we will instead use a type familiar to constructive mathematicians: the Cantor space.

\begin{defn}
  The \textbf{Cantor space} is the type $C \defeq (\bN \to\btwo)$.
\end{defn}

The essential property of $C$, for our purposes, is the following.

\begin{lemma}\label{thm:C=C+C}
  Assuming function extensionality, $C\simeq (C+C)$.
\end{lemma}
\begin{proof}
  From left to right, given $c:\bN\to\btwo$ we define $c'(n) \defeq c(n+1)$, and split into cases based on whether $c(0)$ is $0$ or $1$.
  In the former case, we send $c$ to $\inl(c')$, and in the second case we send it to $\inr(c')$.

  From right to left, we send $\inl(c)$ to $c_0$, where $c_0(0)\defeq 0$ and $c_0(n+1) \defeq c(n)$; and similarly we send $\inr(c)$ to $c_1$ where $c_1(0)\defeq 1$ and $c_0(n+1) \defeq c(n)$.
  It is easy to check that these are inverse equivalences.
\end{proof}

We now consider the ``classifying space of the automorphism group of $C$'', starting by defining it.

\begin{asm}
  For the rest of this section we assume both univalence and propositional truncation.
\end{asm}

From this assumption we also get function extensionality.
In fact, as noted earlier, both univalence and propositional truncation separately imply it.

\begin{defn}
  For any $Y:\type$, we define $\baut(Y)\defeq \sum_{Z:\type}\brck{Z=Y}$.
\end{defn}

Because $\brck{Z=Y}$ is a mere proposition, if we have $(Z,e)$ and $(Z',e')$ in $\baut(Y)$, the type $(Z,e)=(Z',e')$ is equivalent to $Z=Z'$ and hence (by univalence) to $Z\simeq Z'$.
This justifies abusing the notation by identifying an element of $\baut(Y)$ with its first component, which is a type $Z$ that comes equipped with an element of $\brck{Z=Y}$.

In particular, we have the canonical element $(Y,\bproj{\refl_Y}):\baut(Y)$, and the type $(Y,\bproj{\refl_Y})=(Y,\bproj{\refl_Y})$ (the ``loop space'' of $\baut(Y)$ at this ``basepoint'') is equivalent to $Y\simeq Y$, the type of automorphisms of $Y$.
It is in this sense that $\baut(Y)$ is a classifying space for the automorphism group of $Y$.
(The type $\baut(Y)$ is also a classifying space in another sense: to give a map $A\to \baut(Y)$ is equivalent to giving a map $p:B\to A$ such that every fiber is merely equivalent to $Y$, i.e.\ for each $a:A$ we have $\brck{Y = \sum_{b:B} (p(b)=a)}$.)

Our incoherent pre-idempotent will live on the type $X \defeq \baut(C)$, where $C$ is the Cantor space.
Thus, we must next construct a particular map $f:X\to X$.
If we translated Lurie's construction directly, we would do this by first defining an automorphism $F$ of the group $\aut(C)$, by sending an automorphism $h$ to the automorphism $F(h)$ defined as the composite
\[ C \simeq C+C \xrightarrow{h + \id} C+C \simeq C\]
% \[ F(h)(c)(n) \defeq
% \begin{cases}
%   0 &\quad c(0)=0 \text{ and } n=0\\
%   h(\lambda n. c(n+1))(k) &\quad c(0)=0 \text{ and } n=k+1\\
%   c(n) &\quad c(0)=1
% \end{cases}
% \]
where the equivalences come from \cref{thm:C=C+C}.
Then we would use the fact that an automorphism of a group induces an automorphism of its classifying space to obtain from $F$ an automorphism of $\baut(Y)$.

However, although this fact is standard in homotopy theory, it is not obvious from our definition of $\baut(Y)$ in type theory that any automorphism of the group $\aut(Y)$ induces an automorphism of the type $\baut(Y)$.
It can be deduced from the alternative construction of classifying spaces in~\cite{lf:emspaces}; but fortunately in our case there is a better approach.

The univalence axiom has allowed us to define $\baut(Y)$ in such a way that its elements literally \emph{are} types that are merely equivalent to $Y$ (or more precisely, equipped with such a mere equivalence).
Thus, we can define $f$ to act directly on such types, rather than indirectly on their automorphisms.
Specifically, if we define $f(Z) \defeq Z+C$, then the \emph{induced} action on automorphisms will automatically have the intended effect as shown above.

All we have to do is verify that this definition indeed defines an endomorphism of $\baut(C)$, i.e.\ that if $\brck{Z=C}$ then also $\brck{Z+C=C}$.
By the induction principle of propositional truncation, it suffices to prove that if $Z=C$ then $Z+C=C$, and by univalence it suffices to prove that if $e:Z\simeq C$ then $Z+C \simeq C$.
But for this we have the composite
\[ Z+C \;\simeq\; C+C \;\simeq\; C \]
where the first equivalence is $e+\id$ and the second is \cref{thm:C=C+C}.

Next, we have to construct a witness of pre-idempotency for $f$, i.e.\ we must show that for any $Z:\baut(C)$ we have $(Z+C)+C = Z+C$.
We again apply univalence and then use the following composite equivalence:
\begin{equation}
  (Z+C)+C \;\simeq\; Z+(C+C) \;\simeq\; Z+C\label{eq:bautC-preidem}
\end{equation}
consisting of the associativity of coproducts together with \cref{thm:C=C+C}.

We are now ready for the central theorem of this section.

\begin{thm}\label{thm:pre-notsplit}
  There exists a pre-idempotent on $X\defeq \baut(C)$ that does not split.
\end{thm}
\begin{proof}
  The construction of the pre-idempotent is as above; it remains to show that $f$ does not split.
  Lurie's argument is that if it did, then the colimit $X \xrightarrow{f} X \xrightarrow{f} X \xrightarrow{f} \cdots$ would be its splitting, and hence the map from $X$ to that colimit would be surjective on fundamental groups; whereas $f$ itself is certainly not surjective on fundamental groups and so this is impossible.
  In type theory, colimits are difficult to work with, though homotopy type theory with higher inductive types makes them more tractable than otherwise.
  However, we can fortunately again give a more direct argument, based on our concrete construction of $\baut(C)$.
  
  Suppose for contradiction that $f$ is split.
  Then by \cref{thm:split-quasi} it is quasi-idempotent, with witnesses $I$ and $J$.
  For any $Z:\baut(C)$, we have
  \[J(Z) : \ap_f(I(Z)) =_{f(f(f(Z))) = f(f(Z))} I(f(Z)).\]
  Since $f(f(f(Z))) = Z+C+C+C$ and $f(f(Z))=Z+C+C$, by univalence and function extensionality, $J(Z)$ may equivalently be regarded as a homotopy between two specified equivalences $(Z+C+C+C) \to (Z+C+C)$.

  The first of these equivalences (corresponding to $\ap_f(I(Z))$) decomposes the domain and codomain as $(Z+C+C)+C$ and $(Z+C)+C$, mapping the first summand $Z+C+C$ to $Z+C$ by $I(Z)$ and the second summand $C$ to $C$ by the identity.
  As for the second equivalence, if $I$ were the witness~\eqref{eq:bautC-preidem} that we gave above, then the equivalence $Z+C+C+C \to Z+C+C$ corresponding to $I(f(Z))$ would bracket the domain instead as $(Z+C)+(C+C)$ and map it to $(Z+C)+C$ by the identity on $Z+C$ and the ``fold'' equivalence $C+C\to C$.
  Thus, the two could not possibly be homotopic, since they would send the third summand of $Z+C+C+C$ to different summands of the codomain.

  This argument doesn't quite work as stated, since $I$ might \emph{not} be the same proof of idempotency that we gave above.
  (Remember that we are supposing only that $f$ is split, hence quasi-idempotent, in \emph{some} way, since the claim to prove is that $f$ is not split, which makes no reference to any previously existing witness of pre-idempotence.)
  However, whatever $I$ is, it is defined ``for all $Z:\baut(C)$''.
  This implies that the induced equivalences $Z+C+C\to Z+C$ must be \emph{natural} with respect to equivalences between $Z$s (this is not exactly the same sort of naturality that we mentioned in \cref{sec:notation}, but it follows similarly).
  In other words, for any $Z,Z':\baut(C)$ and equivalence $e:Z\simeq Z'$, the following square must commute (up to homotopy):
  \begin{equation}\label{eq:bautC-qidem-nat}
  \vcenter{\xymatrix@C=3pc{
      Z+C+C\ar[r]^-{I(Z)}\ar[d]_{e+\id+\id} &
      Z+C\ar[d]^{e+\id}\\
      Z'+C+C\ar[r]_-{I(Z')} &
      Z'+C
      }}
  \end{equation}
  In particular, we can take $Z$ and $Z'$ to be both $f(C)$, i.e.\ $C+C$, and let $e$ be the ``flip'' automorphism $C+C \simeq C+C$ that interchanges the summands.
  Then the horizontal maps in~\eqref{eq:bautC-qidem-nat} are both the equivalence $(C+C)+C+C \to (C+C)+C$ induced by $I(f(C))$, and~\eqref{eq:bautC-qidem-nat} itself becomes
  \begin{equation}\label{eq:bautC-qidem-nat}
  \vcenter{\xymatrix@C=3pc{
      C+C+C+C\ar[r]^-{I(f(C))}\ar[d]_{e+\id+\id} &
      C+C+C\ar[d]^{e+\id}\\
      C+C+C+C\ar[r]_-{I(f(C))} &
      C+C+C
      }}
  \end{equation}
  Consider elements of the third and fourth summands in the upper-left corner, which are fixed by $e+\id+\id$ on the left.
  Since the two horizontal maps are both $I(f(C))$, it must be that the image of any such element under $I(f(C))$ is fixed by $e+\id$ on the right.
  But the only elements of $C+C+C$ fixed by $e+\id$ are those in the third summand.
  Thus, $I(f(C))$ must map the last two summands in the domain to the last one summand in the codomain, just as our original witness of pre-idempotency did, so our previous argument to a contradiction kicks in.
  % Further details can be found in the formalization.
\end{proof}

Note that we actually showed a bit more: there is a pre-idempotent on $\baut(C)$ that is not even quasi-idempotent.
Because univalence and propositional truncation are consistent assumptions, we conclude:

\begin{corollary}
  It is impossible to prove in \mltt\ that all pre-idempotents split, or even that all pre-idempotents are quasi-idempotent.\qed
\end{corollary}

\section{All quasi-idempotents split}
\label{sec:all-quasi-do}

We now show, in contrast to \cref{thm:pre-notsplit}, that any \emph{quasi}-idempotent \emph{can} be split, assuming nothing more than function extensionality.
There is an obvious naive thing to try: just as $J$ extends $I$ with an additional coherence, we might try to extend the type $\sum_{x:X}(f(x)=x)$ that worked sometimes in \cref{sec:some-split} with an additional coherence, defining
\begin{equation}
  \textstyle \sum_{x:X} \sum_{p:f(x)=x} (\ap_f(p) = I(x)).\label{eq:qidem-try}
\end{equation}
However, Kraus has shown that this does not work in general.

\begin{ex}[Kraus]
  Let $X$ be a type with an element $x_0:X$, and define $f:X\to X$ by $f(x)\defeq x_0$ for all $x:X$.
  Then $f$ is quasi-idempotent with $I(x)\defeq \refl_{x_0}$ and $J(x) \defeq \refl_{\refl_{x_0}}$ for all $x$.
  However, the type of~\eqref{eq:qidem-try} in this case becomes
  \[ \textstyle\sum_{x:X} \sum_{p:x_0=x} (\refl_{x_0} = \refl_{x_0}). \]
  This is equivalent to
  \[ \textstyle \sum_{q:\sum_{x:X} (x_0=x)} (\refl_{x_0} = \refl_{x_0}) \]
  and thence to simply $\refl_{x_0} = \refl_{x_0}$, since the type ${\sum_{x:X} (x_0=x)}$ is contractible.

  On the other hand, $f$ has an evident splitting with $A\defeq \bone$, where $s$ picks out the point $x_0$.
  Thus, if~\eqref{eq:qidem-try} were also a splitting of it, then by \cref{thm:split-uniq} it would be equivalent to $\bone$, i.e.\ contractible.

  However, assuming univalence, there are pointed types $(X,x_0)$ for which $\refl_{x_0} = \refl_{x_0}$ is not contractible.
  For instance, we can take $X$ to be the universe $\type$, with $x_0\defeq \baut(\btwo)$.
  In this case, a nontrivial element of $\refl_{x_0} = \refl_{x_0}$ is constructed in the proof of~\cite[Theorem 4.1.3]{hottbook} (we will give a slightly different construction of the same element in \cref{thm:center-bautbool}).
  Thus, \mltt\ cannot prove that~\eqref{eq:qidem-try} always splits a quasi-idempotent.
\end{ex}

Thus thwarted in our na\"\i{}ve attempts, we turn again to \oo-category theory.
The proof in~\cite[Lemma 7.3.5.14]{lurie:ha} shows that one extra coherence datum suffices to construct a splitting as the colimit of the infinite sequence
\[ X \xrightarrow{f} X \xrightarrow{f}X \xrightarrow{f} \cdots \]
As observed before, colimits are difficult to handle in type theory.
Fortunately, idempotents are completely self-dual, so we might just as well consider the limit of the infinite sequence
\[ \cdots \xrightarrow{f} X \xrightarrow{f}X \xrightarrow{f} X. \]
This is easy to define in type theory: it is $\sum_{a:\bN\to X} \prod_{n:\bN} (f(a_{n+1}) = a_n)$.
Here we see the need for function extensionality: this type involves functions, and we need to construct equalities in it to exhibit it as a retract of $X$.
For reference, we record exactly how to construct equalities in this type.

\begin{lemma}\label{thm:path-split-idem}
  Given $(a,\alpha)$ and $(b,\beta)$ in $\sum_{a:\bN\to X} \prod_{n:\bN} (f(a_{n+1}) = a_n)$, to show that they are equal (assuming function extensionality) it is necessary and sufficent to
  \begin{enumerate}
  \item Construct for each $n:\bN$ an equality $\xi_n : a_n = b_n$, and
  \item Show that for each $n:\bN$ the following diagram of equalities commutes:
    \begin{equation}
      \vcenter{\xymatrix{
          f(a_{n+1})\ar[r]^-{\alpha_n}\ar[d]_{\ap_f(\xi_{n+1})} &
          a_n \ar[d]^{\xi_n}\\
          f(b_{n+1})\ar[r]_-{\beta_n} &
          b_n,
        }}
    \end{equation}
    i.e.\ that $\alpha_n \ct \xi_n = \ap_f(\xi_{n+1}) \ct \beta_n$.
  \end{enumerate}
\end{lemma}
\begin{proof}
  A straightforward application of the results of~\cite[Chapter 2]{hottbook}.
\end{proof}

Now we can prove the main theorem of this section.

\begin{thm}\label{thm:qidem-splits}
  Assuming function extensionality, any quasi-idempotent splits.
\end{thm}
\begin{proof}
  Given $(f,I,J)$, define $A \defeq\sum_{a:\bN\to X} \prod_{n:\bN} (f(a_{n+1}) = a_n)$ as above.
  We define $s:A\to X$ by $s(a,\alpha) \defeq a_0$, and $r:X\to A$ by the slightly less obvious formula
  \[ r(x) \defeq (\lambda n. f(x), \lambda n. I(x)).\]
  (Note that both components of $r(x)$ are actually constant functions, i.e.\ independent of $n$.)
  Now we obviously have $s\circ r = f$; the tricky part is proving $r\circ s = 1$.

  Let $(a,\alpha):A$, so that $a:\bN\to X$ and $\alpha : \prod_{n:\bN} (f(a_{n+1})=a_n)$.
  We must show $(a,\alpha) = r(s(a,\alpha))$.
  By definition, both components of $r(s(a,\alpha))$ are constant, the first at $f(a_0)$ and the second at $I(a_0)$; thus we need a family of equalities $a_n = f(a_0)$ that satisfy commutativity relations.
  For convenience, we break this down into two steps, by defining an intermediate element $(b,\beta):A$ and showing that $(b,\beta)=(a,\alpha)$ and also $(b,\beta)=r(s(a,\alpha))$.
  The definition is
  \begin{alignat*}{2}
    b_n &\defeq f(f(a_{n+1})) && \quad : X\\
    \beta_n &\defeq \ap_{f\circ f}(\alpha_{n+1}) &&\quad : f(f(f(a_{n+2}))) = f(f(a_{n+1}))
  \end{alignat*}

  To show that $(b,\beta)=(a,\alpha)$, we apply \cref{thm:path-split-idem} with
  \[ \xi_n \defeq I(a_{n+1}) \ct \alpha_n \quad : b_n = a_n. \]
  We thus have to show that
  \[ \ap_{f\circ f}(\alpha_{n+1}) \ct I(a_{n+1}) \ct \alpha_n =
  \ap_f(I(a_{n+2}) \ct \alpha_{n+1}) \ct \alpha_n
  \]
  which (after cancelling $\alpha_n$) we can do as follows:
  \begin{alignat*}{2}
    \ap_{f\circ f}(\alpha_{n+1}) \ct I(a_{n+1})
    &= I(f(a_{n+2})) \ct \ap_f(\alpha_{n+1}) &&\quad\text{(naturality)}\\
    &= \ap_f(I(a_{n+2})) \ct \ap_f(\alpha_{n+1})  &&\quad \text{(by }J(a_{n+2}))\\
    &= \ap_f(I(a_{n+2}) \ct \alpha_{n+1}) &&\quad\text{(functoriality)}
  \end{alignat*}\medskip

\noindent  Next we have to show that $(b,\beta)=r(s(a,\alpha))$.
  By definition, $r(s(a,\alpha)) \jdeq (\lambda n. f(a_0), \lambda n. I(a_0))$.
  Invoking \cref{thm:path-split-idem} again, we need to firstly construct $\xi_n : f(f(a_{n+1})) = f(a_0)$ for all $n$.
  We do this by induction on $n$.
  The base case $n\jdeq 0$ is simply $\ap_f(\alpha_0) : f(f(a_1)) = f(a_0)$, while
  the induction step is the composite
  \[ f(f(a_{n+2})) = f(a_{n+1}) = f(f(a_{n+1})) = f(a_0) \]
  of $\ap_f(\alpha_{n+1})$ and $I(a_{n+1})^{-1}$ with the induction hypothesis.
  
  It remains to show that $\ap_{f\circ f}(\alpha_{n+1}) \ct \xi_n = \ap_f(\xi_{n+1}) \ct I(a_0)$ for all $n$, and we do this by induction on $n$ as well.
  For the base case $n\jdeq 0$, this means to check that
  \[ \ap_{f\circ f}(\alpha_{1}) \ct \ap_f(\alpha_0) =
  \ap_f(\ap_f(\alpha_{1}) \ct I(a_{1})^{-1} \ct \ap_f(\alpha_0)) \ct I(a_0). \]
  Applying functoriality of $\ap_f$ on the right, canceling a copy of $\ap_{f\circ f}(\alpha_{1})$ on both sides, and rearranging a little this becomes
  \[ \ap_f(I(a_1)) \ct \ap_f(\alpha_0) = \ap_{f\circ f}(\alpha_0) \ct I(a_0). \]
  But using $J$ we can make this into
  \[ I(f(a_1)) \ct \ap_f(\alpha_0) = \ap_{f\circ f}(\alpha_0) \ct I(a_0). \]
  which is an instance of naturality for $I$.

  Finally, for the induction step, our inductive hypothesis is that the following diagram commutes:
  \begin{equation*}
    \vcenter{\xymatrix@C=12pc{
        f(f(f(a_{n+2})))\ar[r]^{\ap_f(\ap_f(\alpha_{n+1}) \ct I(a_{n+1})^{-1} \ct \xi_n)}
        \ar[d]_{\ap_{f\circ f}(\alpha_{n+1})} &
        % f(f(a_{n+1})) \ar[r]^-{\ap_f(I(a_{n+1}))^{-1}} &
        % f(f(f(a_{n+1}))) \ar[r]^-{\ap_f(\xi_n)} &
        f(f(a_0)) \ar[d]^{I(a_0)}\\
        f(f(a_{n+1}))\ar[r]_-{\xi_n} &
        f(a_0)
      }}
  \end{equation*}
  and our goal is (after applying functoriality of $\ap_f$) to prove that the outer boundary of the following diagram commutes.
  \begin{equation*}
  \vcenter{\xymatrix@R=4pc{
      f(f(f(a_{n+3}))) \ar[rr]^-{\ap_{f\circ f}(\alpha_{n+2})} \ar[d]_{\ap_{f\circ f}(\alpha_{n+2})} &&
      f(f(a_{n+2}))\ar[rr]^-{\ap_f(I(a_{n+2}))^{-1}}
      \ar@(dr,dl)[rr]_{I(f(a_{n+2}))^{-1}}
      \ar@{}[d]|{\text{(nat)}}&
      \ar@{}[d]|(.15){(J)} &
      f(f(f(a_{n+2})))
      \ar[d] %^{\ap_{f\circ f}(\alpha_{n+1})}
      \ar@{}[dr]|{\text{(IH)}}
      \ar[r] & %^{\ap_{f\circ f}(\alpha_{n+1})} &
      %\bullet % f(f(a_{n+1}))
      %\ar[r] & %^-{\ap_f(I(a_{n+1}))^{-1}} &
      %\bullet %f(f(f(a_{n+1})))
      %\ar[r] & %^-{\ap_f(\xi_n)} &
      f(f(a_0)) \ar[d]^{I(a_0)}\\
      f(f(a_{n+2}))\ar[rr]_-{\ap_f(\alpha_{n+1})} \ar[urr]^{\refl} &&
      f(a_{n+1})\ar[rr]_-{I(a_{n+1})^{-1}} &&
      f(f(a_{n+1}))\ar[r]_-{\xi_n} &
      f(a_0)
      }}
  \end{equation*}
  The square marked (IH) is just the inductive hypothesis, while what remains can be filled in by naturality and a further application of $J$.
  % (A proof assistant was very helpful in automatically calculating the hypotheses and goals here.)
\end{proof}

\begin{rmk}\label{thm:split-assumptions}
  Note that the type $A$ and the section $s:A\to X$ involved in the spltting can be defined without knowing either $I$ or $J$, while the retraction $r:X\to A$ and the homotopy $K : s\circ r \sim f$ require only $I$.
  It is only the other homotopy $H:r\circ s \sim \id_A$ that requires the extra coherence datum $J$, and likewise only this homotopy that requires function extensionality.
\end{rmk}

\section{Splitting is a retraction}
\label{sec:splitting-retraction}

\begin{asm}
  In this section we assume the univalence axiom.
\end{asm}

Recall from the introduction that we can define the types of retractions of $X$ and of quasi-idempotents on $X$:
\begin{align*}
  \Retr(X) &\defeq\textstyle \sum_{A:\type} \sum_{r:X\to A} \sum_{s:A\to X} \prod_{a:A} (r(s(a))=a) \\
  \QIdem(X) &\defeq  \textstyle \sum_{f:X\to X} \sum_{I : f\circ f \sim f} \prod_{x:X} (\ap_f(I(x)) = I(f(x)))
\end{align*}
Now \cref{thm:qidem-splits} defines a map
\[ \split : \QIdem(X) \to \Retr(X) \]
and \cref{thm:split-quasi} defines a map
\[ \uli : \Retr(X) \to \QIdem(X). \]
We will now prove the following theorem.

\begin{thm}\label{thm:splitting-retracts}
  $\split$ and $\uli$ exhibit $\Retr(X)$ as a retract of $\QIdem(X)$.
  In other words, $\split\circ\uli=\id_{\Retr(X)}$.
\end{thm}

Before proving this, however, we need to know how to construct equalities in $\Retr(X)$.

\begin{lemma}\label{thm:path-retr}
  Suppose given $(A,r,s,H)$ and $(A',r',s',H')$ in $\Retr(X)$.
  To give an equality $(A,r,s,H)=(A',r',s',H')$, it suffices to give
  \begin{enumerate}
  \item An equivalence $A\simeq A'$, including functions $g : A \to A'$ and $h : A' \to A$ and homotopies $\eta:g h \sim \id$ and $\epsilon: h g \sim \id$;
  \item A homotopy $P : h r' \sim r$;
  \item A homotopy $Q : s' g \sim s$; and
  \item A witness that $\ap_h (H'(g(a))) \ct \epsilon_a = P(s'(g(a))) \ct \ap_r(Q(a)) \ct H(a)$ for all $a:A$.
  \end{enumerate}
\end{lemma}
\begin{proof}
  This should be considered a straightforward application of the results of~\cite[Chapter 2]{hottbook} characterizing identity types of types obtained from different type-formers.
  Since $\Retr(X)$ is a triple $\Sigma$-type, by~\cite[Theorem 2.7.2]{hottbook} we can first decompose equalities in $\Retr(X)$ as quadruples of equalities in its constituent types.
  We then apply univalence to obtain an equivalence $A\simeq A'$ as the first component,~\cite[Lemma 2.9.6]{hottbook} to obtain $P$ and $Q$ as the second and third, and similarly for the fourth component.
  The details are tedious, so we leave them to the reader; like the rest of the paper they have been formalized in Coq.
\end{proof}

\begin{proof}[Proof of \cref{thm:splitting-retracts}]
  Suppose given a retraction $(A, r:X\to A, s:A\to X, H:r\circ s \sim 1)$; we want to show that it is equivalent to the splitting of the induced quasi-idempotent $s r : X \to X$.
  The latter is a new retraction $(A',r',s',H')$ such that $s r = s' r'$ (in fact this equality holds judgmentally).
  By \cref{thm:split-uniq}, we have an equivalence $A\simeq A'$ composed of $g = r' s : A \to A'$ and $h = r s' : A' \to A$, with $\eta : g h \jdeq r' s r s' = r' s' r' s' = \id$ and dually $\epsilon : h g \jdeq r s'  r' s = r s r s = \id$.
  Moreover, we have $P : h r' \jdeq r s' r' = r s r = r$ and $Q : s' g \jdeq s' r' s = s r s = s$; thus it remains to construct the fourth datum in \cref{thm:path-retr}.

  In general, both sides of this equality are homotopies $r s' r' s' r' s \sim 1$.
  Note that in our case, the domain $r s' r' s' r' s$ is judgmentally equal to $rsrsrs$.
  Substituting the definitions of $r'$, $s'$, and $H'$ from \cref{thm:qidem-splits}, we see that the left-hand side of the desired equality is the composite
  \begin{multline}\label{eq:splretr1}
    rsrsrsa \xrightarrow{\ap_{rsrs}(H(rsa))^{-1}} rsrsrsrsa \xrightarrow{\ap_{rs}(H(rsrsa))} rsrsrsa \\ \xrightarrow{\ap_{rs}(H(rsa))} rsrsa \xrightarrow{H(rsa)} rsa \xrightarrow{Ha} a
  \end{multline}
  while the right-hand side is the composite
  \begin{equation}
    rsrsrsa \xrightarrow{H(rsrsa)} rsrsa \xrightarrow{\ap_{rs}(Ha)} rsa \xrightarrow{Ha} a.\label{eq:splretr2}
  \end{equation}
  Now by naturality, we have
  \[{\ap_{rs}(H(rsrsa))} \ct {\ap_{rs}(H(rsa))} = {\ap_{rsrs}(H(rsa))} \ct {\ap_{rs}(H(rsa))}.\]
  Applying this in the middle of~\eqref{eq:splretr1}, and canceling ${\ap_{rsrs}(H(rsa))}$ with its inverse on the left, we reduce it to
  \[rsrsrsa \xrightarrow{\ap_{rs}(H(rsa))} rsrsa \xrightarrow{H(rsa)} rsa \xrightarrow{Ha} a.\]
  Now naturality gives ${\ap_{rs}(H(rsa))} \ct {H(rsa)} = H(rsrsa) \ct {H(rsa)}$, so this is equal to 
  \[rsrsrsa \xrightarrow{H(rsrsa)} rsrsa \xrightarrow{H(rsa)} rsa \xrightarrow{Ha} a. \]
  Comparing this to~\eqref{eq:splretr2}, we can cancel $H(rsrsa)$ on the left, reducing the problem to ${\ap_{rs}(Ha)}\ct Ha = {H(rsa)} \ct Ha$, which is another naturality.
\end{proof}

\cref{thm:splitting-retracts} makes no reference to a specified function $f$, but we can deduce from it a statement that does.
Given an endofunction $f$, we define a \textbf{splitting of $f$} to be a retraction $(A,r,s,H):\Retr(X)$ together with a homotopy $K:s\circ r \sim f$.
These form a type $\Split(X,f)\defeq \sum_{(A,r,s,H):\Retr(X)} (s\circ r \sim f)$.
We also have a type $\QIdem(X,f)\defeq \sum_{(I:f\circ f \sim f)} \prod_{x:X} (\ap_f(I(x)) = I(f(x)))$ of ``quasi-idempotence data for $f$''.

\begin{corollary}\label{thm:split-retract}
  For any $f:X\to X$, the type $\Split(X,f)$ is a retract of $\QIdem(X,f)$.
\end{corollary}
\begin{proof}
  Let $k:\Retr(X) \to (X\to X)$ take $(A,r,s,H)$ to the composite $s r$.
  Then $\Split(X,f)$ is, by definition, the \emph{fiber} of $k$ over $f$.
  On the other hand, by~\cite[Lemma 4.8.1]{hottbook}, the type $\QIdem(X,f)$ is equivalent to the fiber over $f$ of the first projection $\QIdem(X) \to (X\to X)$.
  The retraction from \cref{thm:splitting-retracts} commutes with these maps to $X\to X$; hence by~\cite[Lemma 4.7.3]{hottbook}, it induces a retraction between their fibers.
\end{proof}

Similarly, we can consider the case when $f$ is already equipped with a witness $I$ of pre-idempotency.
We define $\Split(X,f,I)$ to be
\[ \textstyle \sum_{(A,r,s,H,K):\Split(X,f)} \prod_{x:X} (\ap_f(K(x))^{-1} \ct K(s(r(x)))^{-1} \ct \ap_s(H(r(x))) \ct K(x) = I(x)), \]
the long composite being just the result of transferring the definition of \cref{thm:split-pre} across the homotopy $K : s r \sim f$.
And of course we have the type $\QIdem(X,f,I)\defeq \prod_{x:X} (\ap_f(I(x)) = I(f(x)))$ of quasi-idempotence enhancements of $I$.

\begin{corollary}\label{thm:split-retract-I}
  For any $(f,I)$, the type $\Split(X,f,I)$ is a retract of $\QIdem(X,f,I)$.
\end{corollary}
\begin{proof}
  As in \cref{thm:split-retract}, we take fibers of two maps to the type $f\circ f \sim f$ of $I$.
  We leave the details to the reader; or they can be found in the formalization.
\end{proof}

\section{Splitting is not an equivalence}
\label{sec:splitting-not-equiv}

We now consider what can be said about the composite $\uli\circ\split$, which is an endofunction of $\QIdem(X)$.
Our first observation is that it preserves the witness $I$ of \emph{pre}-idempotence.

\begin{thm}\label{thm:I-recovers}
  Assume function extensionality.
  Then given a quasi-idempotent $(f,I,J)$, if we split it as in \cref{thm:qidem-splits}, the witness of pre-idempotence induced from the splitting as in \cref{thm:split-pre} is equal to $I$.
\end{thm}
\begin{proof}
  Given a retraction $s:A\to X$ and $r:X\to A$ with $H : r\circ s \sim 1$, the induced $I$ was defined in \cref{thm:split-pre} by $I(x)\defeq \ap_s (H(r(x)))$.
  For the splitting from \cref{thm:qidem-splits} with $A \defeq \sum_{a:\bN\to X} \prod_{x:X}(f(a_{n+1})=a_n)$, we have $s(a,b)\defeq a_0$, so the induced $I'(x)$ is just the $0$-component of the homotopy $H:r\circ s \sim 1$ at $r(x) \defeq (\lambda n. f(x), \lambda n.I(x))$.
  By construction, this is the composite
  \[ f(f(x)) \xrightarrow{\ap_f(I(x))^{-1}} f(f(f(x))) \xrightarrow{I(f(x))} f(f(x)) \xrightarrow{I(x)} f(x)\]
  where $I$ is the given witness of pre-idempotence.
  But by the given $J$, we have $\ap_f(I(x)) = I(f(x))$, so this reduces to just $I(x)$.
\end{proof}

Thus, if the further coherence witness $J$ were also recovered from the splitting, we would have $\uli\circ\split = \id$, and hence (assuming univalence, so that the results of the previous section apply) $\split$ and $\uli$ would be inverse \emph{equivalences} between $\Retr(X)$ and $\QIdem(X)$.
By \cref{thm:split-retract} and \cref{thm:split-retract-I}, this would also yield equivalences between $\Split(X,f)$ and $\QIdem(X,f)$, and between $\Split(X,f,I)$ and $\QIdem(X,f,I)$, for any $f$ and $I$.
We will show that this is impossible in general, beginning with the following observation.

\begin{lemma}
  Assuming univalence, if $f \defeq \id_X$ and $I(x)\defeq \refl_{x}$ for all $x$, then the type $\Split(X,f,I)$ is contractible.
\end{lemma}
\begin{proof}
  Recall that for any type $B$ and point $b_0:B$, the type $\sum_{b:B} (b=b_0)$ is contractible.
  By univalence, it follows that for any type $X$, the type $\sum_{A:\type} (A\simeq X)$ is contractible.
  Since $\Split(X,f,I)$ begins with a $\sum_{A:\type}$, it will suffice to show that the rest of it is equivalent to $(A\simeq X)$.

  We will use the ``half-adjoint equivalence'' definition of $(A\simeq X)$ from~\cite[\S4.2]{hottbook}.
  The data $r$ and $s$ are, of course, maps back and forth, while since $f\jdeq \id_X$ the data $H$ and $K$ have the right types to be the homotopies $\epsilon$ and $\eta$.
  It remains, therefore, to show that the type of the remaining datum:
  \[ \textstyle \prod_{x:X} (\ap_f(K(x))^{-1} \ct K(s(r(x)))^{-1} \ct \ap_s(H(r(x))) \ct K(x) = I(x))\]
  is equivalent to $\prod_{a:A} (\ap_s(H(a)) = K(s(a)))$.
  Now since $f\jdeq \id_X$ and $I(x)\jdeq \refl_x$, we can discard the $\ap_f$, move the $K(x)^{-1}$ to the other side, and then cancel it.
  If we move $K(s(r(x)))$ to the other side as well, we obtain $\prod_{x:X} (\ap_s(H(r(x))) = K(s(r(x))))$.
  Finally, since $s$, $H$, and $K$ suffice to show that $r$ is an equivalence, we can transport along it to obtain the desired type $\prod_{a:A} (\ap_s(H(a)) = K(s(a)))$.
\end{proof}

Therefore, if we had $\uli\circ\split = \id$, then $\QIdem(X,\id_X,\lambda x.\refl_x)$ would also be contractible for any $X$.
However, $\QIdem(X,\id_X,\lambda x.\refl_x)$ reduces to $\prod_{x:X} (\refl_x = \refl_x)$, which we might call the \textbf{2-center} of $X$ (see \cref{sec:baut-baut-bool} for why).
Thus, it suffices to construct a type $X$ whose 2-center has nontrivial inhabitants.
Of course, such an $X$ cannot be a set or even a 1-type, but it will suffice for it to be a 2-type (i.e.\ its twice-iterated equality types $p =_{(x=_X y)} q$ are sets).

\begin{rmk}
  As pointed out by a referee, there are many ways to construct such a 2-type using higher inductive types.
  For instance, if $X$ is the 2-truncation of the 2-sphere, we can define an element of $\prod_{x:X} (\refl_x = \refl_x)$ by truncation-induction (since $\refl_x = \refl_x$ is a 0-type) followed by sphere-induction, sending the basepoint to the generating 2-loop and the rest being trivial for truncation reasons.
  More generally, we could take $X$ to be an Eilenberg--Mac Lane space $K(G,2)$ for any nontrivial abelian group $G$ (see~\cite{lf:emspaces}) --- the 2-truncation of the 2-sphere is a $K(\mathbb{Z},2)$.
  However, if we are willing to work a little harder, we can obtain such a 2-type using only univalence and propositional truncation: just as $\baut(\btwo)$ supports a nontrivial element of the \textbf{1-center} $\prod_{x:X} (x=x)$, to find a nontrivial element of the 2-center we can use $X\defeq \baut(\baut(\btwo))$.
\end{rmk}

\begin{thm}\label{thm:2center-bautbautbool}
  Assuming univalence and propositional truncation, if $X\defeq \baut(\baut(\btwo))$, then $\prod_{x:X} (\refl_x = \refl_x)$ has a nontrivial element.
\end{thm}
\begin{proof}[Idea of proof]
  As an \oo-groupoid, $\baut(\btwo)$ has one object with two automorphisms, the identity and the flip.
  Since automorphisms preserve identities, $\baut(\btwo)$ \emph{itself} has only one automorphism, but there are two self-homotopies of that automorphism.
  In other words, the space of automorphisms of $\baut(\btwo)$ is equivalent to $\baut(\btwo)$ itself.
  Thus, $\baut(\baut(\btwo))$ has one object, with only its identity morphism, but two 2-morphisms from that identity to itself.
  This nonidentity 2-morphism is essentially our desired nontrivial element.
  
  However, proving this carefully in type theory requires a lot of lemmas about classifying spaces, so we defer it to the next section.
  (An alternative proof can be found in~\cite[Lemma 7.5.2]{kraus:thesis}.)
\end{proof}

\begin{corollary}
  In \mltt\ with function extensionality (which is necessary to construct the function $\split$), it is impossible to prove that $\uli\circ\split = \id_{\QIdem(X)}$ for every type $X$.\qed
\end{corollary}

\section{The double classifying space of 2}
\label{sec:baut-baut-bool}

Here we will prove \cref{thm:2center-bautbautbool}.
For this we need some preliminary lemmas about types of the form $\baut(X)$.

\begin{asm}
  Throughout this section we assume both univalence and propositional truncation.
\end{asm}

Our first lemma says that defining a section of a family of sets indexed by $\baut(X)$ is equivalent to giving an element lying over $X$ itself which is fixed by all automorphisms of $X$.
To make sense of ``fixed by'', we use the notion of \textbf{transport}: given any type family $B:A\to \type$, if $p : x=_A y$ we have a function $p_* : B(x) \to B(y)$ defined by identity-type elimination (see~\cite[Chapter 2]{hottbook} for more information).

For convenience, we will frequently implicitly coerce elements of $\baut(X)$ to their underlying types.

\begin{lemma}\label{thm:baut-ind-hset}
  Let $X$ be any type, and suppose $P:\baut(X)\to\type$ is a family of sets.
  Then
  \[ \textstyle
  \Big(\prod_{Z:\baut(X)} P(Z) \Big) \simeq
  \Big(\sum_{e : P(X)} \prod_{g:X=X} g_*(e)=e\Big). \]
\end{lemma}
\begin{proof}
  Since $\baut(X) \defeq \sum_{Z:\type}\brck{Z=X}$, we have
  \begin{equation}
    \textstyle \Big(\prod_{Z:\baut(X)} P(Z)\Big) \simeq
    \Big( \prod_{Z:\type} \big( \brck{Z=X} \to P(Z)\big)\Big).\label{eq:baut-ind-lhs}
  \end{equation}
  Now recall from~\cite[Theorem 5.4]{keca:anon} that if $B$ is a set, then a function $A\to B$ factors through $\brck{A}$ if and only if it is weakly constant.
  In fact, it is not hard to show that when $B$ is a set, the type $\brck{A} \to B$ is equivalent to the type of weakly constant functions $A\to B$.
  Thus, the right-hand-side of~\eqref{eq:baut-ind-lhs} is equivalent to
  \[ \textstyle \prod_{Z:\type} \sum_{f:(Z=X) \to P(Z)} \prod_{p,q:Z=X} (f(p)=f(q)).\]
  Rearranging this with~\cite[Theorem 2.15.7]{hottbook} (the ``type-theoretic axiom of choice''), we obtain
  \[ \textstyle \sum_{f:\prod_{Z:\type} (Z=X) \to P(Z)}
  \prod_{Z:\type} \prod_{p,q:Z=X} (f(Z,p)=f(Z,q)). \]
  Applying the universal property of identity types~\cite[(2.15.10)]{hottbook}, this becomes
  \[ \textstyle \sum_{f:\prod_{Z:\type} (Z=X) \to P(Z)}
  \prod_{p:X=X} (f(X,p)=f(X,\refl_X)). \]
  The same property implies that $\prod_{Z:\type} (Z=X) \to P(Z)$  is equivalent to $P(X)$, where the inverse equivalence sends $e:P(X)$ to $\lambda Z. \lambda q. q_*(e)$.
  Transferring across this equivalence, we obtain the desired result.
\end{proof}

We can use this to characterize types of the form $\prod_{Z:\baut(X)} (Z=Z)$, which are ``one level down'' from the type $\prod_{Z:\baut(X)} (\refl_Z = \refl_Z)$ considered in \cref{thm:2center-bautbautbool}.

\begin{lemma}\label{thm:center-baut}
  If $X$ is a set, then $\prod_{Z:\baut(X)} (Z=Z)$ is equivalent to
  \[ \textstyle \sum_{f:X\simeq X} \prod_{g:X\simeq X} (f\circ g = g \circ f) \]
\end{lemma}
\begin{proof}
  Since $X$ is a set, $Z=Z$ is a set for any $Z:\baut(X)$.
  Thus, by \cref{thm:baut-ind-hset}, $\prod_{Z:\baut(X)} (Z=Z)$ is equivalent to
  \[ \textstyle \sum_{e:X=X} \prod_{g:X= X} (g_*(e) = e). \]
  The result follows by applying~\cite[Theorem 2.11.5]{hottbook} and the univalence axiom.
\end{proof}

\begin{rmk}\label{thm:center-bautbool}
  If $X\defeq\btwo$, it is easy to show that $X$ has precisely two automorphisms, the identity and the flip.
  Since the flip is an involution, it commutes with itself, and of course it commutes with the identity; thus by \cref{thm:center-baut} it yields a nontrivial element of $\prod_{Z:\baut(\btwo)} (Z=Z)$.
  This gives a slightly different proof of~\cite[Theorem 4.1.3]{hottbook}.
  In fact, \cref{thm:center-baut} gives the stronger result that $\prod_{Z:\baut(\btwo)} (Z=Z)$ has \emph{exactly} one nontrivial element (hence in particular our nontrivial element agrees with that of~\cite[Theorem 4.1.3]{hottbook}).
\end{rmk}

\cref{thm:center-baut} says that $\prod_{Z:\baut(X)} (Z=Z)$ is equivalent to the type of automorphisms of $X$ that commute with all other automorphisms of $X$, i.e. the \textbf{center} of $\aut(X)$.
This explains why when we move up a level to the type appearing in \cref{thm:2center-bautbautbool}, we may reasonably call it the \textbf{2-center}.

\begin{lemma}\label{thm:center2-baut}
  If $X$ is a 1-type, then $\prod_{Z:\baut(X)} (\refl_Z = \refl_Z)$ is equivalent to
  \[ \textstyle \sum_{f:\prod_{x:X} (x=x)} \prod_{g:X\simeq X} \prod_{x:X} (\ap_g(f(x)) = f(g(x))). \]
\end{lemma}
\begin{proof}
  Since $X$ is a 1-type, $(\refl_Z=\refl_Z)$ is a set for any $Z:\baut(X)$.
  Thus, by \cref{thm:baut-ind-hset}, $\prod_{Z:\baut(X)} (\refl_Z = \refl_Z)$ is equivalent to
  \[\textstyle \sum_{e:\refl_X = \refl_X} \prod_{g:X=X} g_*(e)=e. \]
  Now by univalence and function extensionality, $\refl_X = \refl_X$ is equivalent to $\prod_{x:X} (x=x)$, while of course $X=X$ is equivalent to $X\simeq X$.
  Under this equivalence, $g_*(e)$ is identified with $\lambda x.\ap_g(f(g^{-1}(x)))$.
  Finally, since $g$ is an equivalence, we can transfer it to the other side of the equation and obtain the desired result.
\end{proof}

We want to apply \cref{thm:center2-baut} to $X\defeq \baut(\btwo)$.
In that case, we have a nontrivial ${f:\prod_{x:\baut(\btwo)} (x=x)}$ from \cref{thm:center-bautbool}.
Therefore, to prove \cref{thm:2center-bautbautbool} it remains to show that this $f$ satisfies $\ap_g(f(Z)) = f(g(Z))$ for all automorphisms $g$ of $\baut(\btwo)$ and all $Z:\baut(\btwo)$.

Of course, this requires knowing something about \emph{all} automorphisms of $\baut(\btwo)$.
In our proof sketch of \cref{thm:2center-bautbautbool}, we claimed that the space of automorphisms of $\baut(\btwo)$ should be equivalent to $\baut(\btwo)$ itself, but our argument involved decomposing an \oo-groupoid into ``objects, morphisms, and 2-morphisms'' which is not possible in homotopy type theory.
Instead, we need to give a more ``synthetic'' argument, analogous to our construction of the incoherent pre-idempotent on $\baut(C)$ in \cref{sec:some-dont}.

The idea is as follows: since $\btwo$ is an abelian group (the cyclic group of order 2), $\baut(\btwo)$ should also be an abelian \oo-group.
% (or 2-group, since it is a 1-type).
Since multiplication by a fixed element of an \oo-group is an equivalence, this will give us a map $\baut(\btwo) \to (\baut(\btwo) \simeq \baut(\btwo))$, which we can then show to be an equivalence.

Now we have to define the group operation on $\baut(\btwo)$ internally.
The idea to keep in mind is that the elements of $\baut(\btwo)$ are the ``finite sets with two elements''.
They are merely isomorphic to $\btwo$, but to \emph{specify} such an isomorphism $\btwo\simeq Z$ is the same as specifying an element of $Z$ (to be the image of $1:\btwo$).

The ``morally-best'' definition of the group operation would perhaps be as a ``tensor product over the field with two elements''.
However, since we are not assuming any colimits, we use instead the following:
\[Z \ast W \defeq (Z\simeq W). \]
Since $\btwo\simeq (\btwo\simeq\btwo)$, it follows that $Z\simeq W$ is in $\baut(\btwo)$ if $Z$ and $W$ are.

This definition is obviously symmetric, $Z\ast W = W\ast Z$.
Moreover, it has $\btwo$ itself as a left (hence also right) identity: if $W:\baut(\btwo)$ then an equivalence $e:\btwo \simeq W$ is uniquely determined by $e(1):W$.
And $Z\ast Z$ is equivalent to $\btwo$ for any $Z$, since it has a canonically specified element (namely the identity); thus in particular $\ast$ has inverses.
The trickiest part is showing associativity.

\begin{lemma}
  For any $Z,W,Y:\baut(\btwo)$ we have $(Z\ast W)\ast Y = Z \ast (W\ast Y)$.
\end{lemma}
\begin{proof}
  Since $\ast$ is symmetric, it suffices to prove $Y\ast (Z\ast W) = Z\ast (Y\ast W)$.
  We will show that for all $Y,Z,W$ there is a map $\sigma : Y\ast (Z\ast W) \to Z\ast (Y\ast W)$, and that this map is its own inverse (when applied with $Y$ and $Z$ switched).

  Now, an element of $Y\ast (Z\ast W)$ can be regarded as a function $e:Y \to (Z\to W)$ with the additional properties that
  \begin{enumerate}
  \item each function $e(y):Z \to W$ is an equivalence, and\label{item:ast1}
  \item $e$ induces an equivalence from $Y$ to $Z\simeq W$.\label{item:ast2}
  \end{enumerate}
  Since being an equivalence is a mere proposition, two elements of $Y\ast (Z\ast W)$ are equal just when their underlying functions $e:Y \to (Z\to W)$ are.

  We will define $\sigma$ so that its action on underlying functions simply swaps arguments: $\sigma(e)(z)(y) = e(y)(z)$.
  Thus, it will automatically be self-inverse.
  What remains is to show that $\sigma(e)$ satisfies~(\ref{item:ast1}) and~(\ref{item:ast2}) assuming $e$ does.

  However, since all of our types are finite sets (that is, they are merely isomorphic to a standard finite type such as $\sum_{k:\mathbb{N}}(k<n)$), a map between them is an equivalence as soon as it is injective.
  Thus, to show~(\ref{item:ast1}) for $\sigma(e)$ we must show that if $e(y)(z) = e(y')(z)$ for some $z:Z$, then $y=y'$.
  But by~(\ref{item:ast2}) for $e$, we have $y = e^{-1}(e(y))$ and $y' = e^{-1}(e(y'))$, so it suffices to show that $e(y)=e(y')$.
  This follows from $e(y)(z) = e(y')(z)$ since an equivalence between 2-element sets is determined by its action on a single element.

  Similarly, to show~(\ref{item:ast2}) for $\sigma(e)$, we must show that if $e(y)(z) = e(y)(z')$ for \emph{all} $y:Y$, then $z=z'$.
  But this in particular implies that $e(y)(z) = e(y)(z')$ for \emph{some} $y$, and thus $z=z'$ by~(\ref{item:ast1}) for $e$.
\end{proof}

Now we can prove that $\baut(\btwo)$ is equivalent to its own automorphism group.

\begin{lemma}\label{thm:aut-bautbool}
  $\baut(\btwo) \simeq (\baut(\btwo)\simeq \baut(\btwo))$.
\end{lemma}
\begin{proof}
  The map from left to right sends $Z$ to $\lambda W. Z\ast W$.
  Since $Z\ast(Z\ast W) = (Z\ast Z)\ast W = \btwo\ast W = W$, the function $\lambda W. Z\ast W$ is an equivalence whose inverse is itself.

  The map from right to left sends $e:\baut(\btwo)\simeq \baut(\btwo)$ to $e^{-1}(\btwo)$.
  The round-trip composite on the left is the identity since $\btwo$ is a unit for $\ast$.
  On the other side, we must show that for any $e:\baut(\btwo)\simeq \baut(\btwo)$ and $W$ we have $e(W) = e^{-1}(\btwo)\ast W$.

  In fact, we will show that $e^{-1}(Z)\ast W = Z\ast e(W)$ for any $Z,W$; the desired result then follows by taking $Z\defeq \btwo$.
  However, by univalence, we have $(e^{-1}(Z)\ast W) \;=\; (e^{-1}(Z) = W)$ and similarly on the other side, and $(e^{-1}(Z) = W) \simeq (Z= e(W))$ holds for any equivalence $e$.
\end{proof}

Finally, we can prove \cref{thm:2center-bautbautbool}.

\begin{thm}
  There is an element of $\prod_{Z:\baut(\baut(\btwo))} (\refl_Z = \refl_Z)$ that is not equal to $\lambda Z.\refl_{\refl_Z}$.
\end{thm}
\begin{proof}
  By \cref{thm:center-bautbool}, we have an ${f:\prod_{x:\baut(\btwo)} (x=x)}$ that is unequal to $\lambda x.\refl_x$.
  Thus, by \cref{thm:center2-baut}, it remains to show that this $f$ satisfies $\ap_g(f(Z)) = f(g(Z))$ for all automorphisms $g$ of $\baut(\btwo)$ and all $Z:\baut(\btwo)$.

  Let $g$ and $Z$ be given.
  By \cref{thm:aut-bautbool}, we may assume $g$ is of the form $\lambda Y. W\ast Y$ for some $W:\baut(\btwo)$.
  And since our goal is a mere proposition, we may assume that $Z$ and $W$ are both $\btwo$.

  Now since $g(Y) \jdeq \btwo\ast Y$ and $\btwo$ is a left unit for $\ast$, we have a homotopy $H:g \sim \id$.
  And by ``dependent $\ap$'' for $f$ (see~\cite[Lemma 2.3.4]{hottbook}) applied to $H_\btwo : \btwo\ast\btwo = \btwo$, we have
  \[ f(\btwo\ast\btwo) \ct H_\btwo = H_\btwo \ct f(\btwo). \]
  Since also $g(Z) \jdeq \btwo\ast\btwo$, what we have to show becomes
  \[ \ap_g(f(\btwo)) \ct H_\btwo = H_\btwo \ct f(\btwo).\]
  However, this is just naturality for $H$.
\end{proof}

\section{Coherent idempotents}
\label{sec:coherent-idempotents}

We have seen that, assuming univalence, $\Retr(X)$ is a retract of $\QIdem(X)$, and in general a nontrivial one.
As remarked in the introduction, in \oo-category theory the ``space of retractions of $X$'' is equivalent to the ``space of fully-coherent idempotents on $X$''.
This follows from~\cite[Corollary 4.4.5.14]{lurie:higher-topoi}.
As stated, that corollary says that in an \oo-category where (fully-coherent) idempotents split, the space of \emph{all} retractions is equivalent to the space of \emph{all} fully-coherent idempotents; but since this equivalence is fibered over the space of objects of the \oo-category itself, it induces fiberwise equivalences for each object $X$.

Thus, in homotopy type theory with the univalence axiom, it is reasonable to expect that $\Retr(X)$ should be equivalent to ``the type of fully-coherent idempotents on $X$'', if we were able to define the latter type.
In particular, since $\Retr(X)$ is not generally equivalent to $\QIdem(X)$, the latter is not a correct definition of the type of fully-coherent idempotents.

As mentioned in the introduction, we could take $\Retr(X)$ as a \emph{definition} of the type of fully-coherent idempotents, but this would suffer from two drawbacks:
\begin{enumerate}
\item It would be aesthetically unsatisfying to say that ``an idempotent'' comes by definition equipped with a splitting.
  Morally, splitting should be something that is \emph{done to} an idempotent.
\item It lives in a higher universe than the type $X$, since it involves a $\sum_{A:\type}$.
\end{enumerate}
Both of these problems can be solved with the following observation:
since $\Retr(X)$ is a retract of $\QIdem(X)$, the composite $\uli \circ \split$ \emph{is a quasi-idempotent} on $\QIdem(X)$.
We can therefore \emph{split it} using the construction of \cref{thm:qidem-splits}.
By \cref{thm:split-uniq}, the resulting type will be equivalent to $\Retr(X)$; but it will live (like $\QIdem(X)$ itself) in the same universe as $X$, and its elements do not obviously contain a splitting.
Thus, we propose the following definition.

\begin{defn}
  A \textbf{(fully-coherent) idempotent} on a type $X$ is an element of the splitting of $\uli \circ \split$.
  Somewhat more explicitly, this type is
  \[ \textstyle \Idem(X) \defeq \sum_{a:\bN \to \QIdem(X)} \prod_{n:\bN} (\uli(\split(a_{n+1})) = a_n). \]
  Similarly, an \textbf{idempotent structure on} $f:X\to X$ is an element of the splitting of the similarly induced idempotent on $\QIdem(X,f)$.
\end{defn}

It is worth thinking a little about what assumptions are necessary for this definition.
It may appear at first to require univalence, since $\uli \circ \split$ is only a (quasi-)idempotent because of \cref{thm:path-retr}, which uses univalence.
However, as observed in \cref{thm:split-assumptions}, to define the splitting \emph{type} of an idempotent does not require the witnesses of quasi-idempotency or pre-idempotency.
Thus, in order to define the type $\Idem(X)$ we really only require function extensionality, since that suffices to define the maps $\uli$ and $\split$.

It is possible, of course, to unwind this definition further, but it becomes quite complicated.
Nevertheless, it is satisfying that we can give \emph{some} correct definition of fully-coherent idempotent, since the general problem of representing fully-coherent higher homotopy structures in type theory is unsolved.

There is an interesting analogy to the situation with equivalences.
The na\"\i{}ve definition of an equivalence (or isomorphism) between types $A$ and $B$ would be
\begin{equation}
  \textstyle \sum_{f:A\to B} \sum_{g:B\to A} (g\circ f \sim \id_A) \times (f\circ g \sim \id_B).\label{eq:qinv}
\end{equation}
However, this gives the wrong homotopy type.
We might then think that we need an infinite tower of further coherences, but in fact it suffices to give one additional datum, although there are several choices for what that extra datum might be (see~\cite[Chapter 4]{hottbook}).

Nevertheless, given an element of~\eqref{eq:qinv}, it is possible to alter one of its constituent homotopies to obtain a fully-coherent equivalence.
This exhibits the type of equivalences as a retract of~\eqref{eq:qinv}, just as our type of idempotents is a retract of the type of quasi-idempotents.
There is a difference, however, in that ``$f$ is an equivalence'' is a mere proposition, whereas ``$f$ is an idempotent'' is not.

\section{Conclusions}
\label{sec:conclusions}

The main result of this paper is that not all idempotents in Martin-L\"{o}f type theory can be proven to split, but if we assume function extensionality then one additional coherence condition suffices to make an idempotent splittable.
In addition to its intrinsic interest, this shows how ideas from homotopy theory and higher category theory can be useful even for the study of non-homotopical type theory.

In the homotopical case, however, there is more to say about idempotents, which can be partially or fully coherent.
Although fully coherent homotopical structures are often difficult to define in type theory, we have managed to define the type of fully coherent idempotents, by splitting an idempotent on the type of partially coherent ones.

With that said, this paper still leaves a number of interesting open questions about idempotents in type theory.

\begin{oprob}
  Can we split quasi-idempotents in \mltt\ without assuming function extensionality?
  In particular, is there any more ``finite'' way to construct such a splitting?
\end{oprob}

\begin{oprob}
  Is the section $\Idem(X) \to \QIdem(X)$ an embedding?
  Equivalently, by \cref{thm:split-splitsupp}, does the type $\uli(\split(f,I,J)) = (f,I,J)$ admit a weakly constant endofunction for every quasi-idempotent $(f,I,J)$?
  I expect the answer is no, but an explicit counterexample would be nice to have.
\end{oprob}

\begin{oprob}
  Similarly, is the induced map from $\Idem(X)$ to the type $\PIdem(X)$ of \emph{pre}-idempotents an embedding?
  Again, I expect the answer is no, but this appears to be an open problem even in \oo-category theory; see~\cite{shulman:splhoidem}.
\end{oprob}

\begin{oprob}
  Can $\Idem(X)$ be defined without assuming even function extensionality?
  More precisely, is there a type we can define without function extensionality that becomes equivalent to $\Idem(X)$ if we assume function extensionality?
\end{oprob}

\begin{oprob}
  Are there any other fully-coherent higher-homotopy structures that can be obtained from a finite amount of coherence by splitting an idempotent?
\end{oprob}

\section*{Acknowledgement}

This paper would not exist without Mart\'in Escard\'o: not just because he asked the original question and contributed many of the results in \cref{sec:some-split}, but because during a long email discussion he provided both encouragement and an indispensable sounding-board for the development of the rest of it, and gave helpful feedback on a draft.

\bibliographystyle{alpha}
\bibliography{all,mo}

\end{document}